\documentclass[10pt]{article}
\usepackage{amsmath}
\usepackage{amssymb}
\usepackage{amsthm}
\usepackage[usenames]{color}
\usepackage{amscd}
\usepackage{dsfont}
\usepackage[colorlinks=true,linkcolor=webgreen,filecolor=webbrown,
citecolor=webgreen]{hyperref}

\definecolor{webgreen}{rgb}{0,.5,0}
\definecolor{webbrown}{rgb}{.6,0,0}

\def\C{{\mathbb{C}}}

\def\N{{\mathbb{N}}}
\def\Z{{\mathbb{Z}}}

\def\1{{\bf 1}}

\def\RE{\operatorname{Re}}
\def\id{\operatorname{id}}
\def\lcm{\operatorname{lcm}}
\def\ds{\displaystyle}

\newtheorem{theorem}{Theorem}

\newtheorem{cor}[theorem]{Corollary}
\newtheorem{appl}[theorem]{Application}
\newtheorem{prop}[theorem]{Proposition}

\begin{document}

\title{Some remarks on a paper of V.~A.~Liskovets}
\author{L\'aszl\'o T\'oth \footnote{The author gratefully acknowledges support from the
Austrian Science Fund (FWF) under the project Nr. P20847-N18.}
 \\ \\ Department of Mathematics, University of P\'ecs \\ Ifj\'us\'ag u. 6, H-7624 P\'ecs,
Hungary \\ and
\\ Institute of Mathematics, Department of Integrative Biology \\
Universit\"at f\"ur Bodenkultur, Gregor Mendel-Stra{\ss}e 33, A-1180
Wien, Austria \\ E-mail: ltoth@gamma.ttk.pte.hu}
\date{}
\maketitle

\begin{abstract} We deduce new properties of the orbicyclic function
$E$ of several variables investigated in a recent paper by
V.~A.~Liskovets. We point out that the function $E$ and its
connection to the number of solutions of certain linear congruences
occur in the literature in a slightly different form. We investigate
another similar function considered by Deitmar, Koyama and Kurokawa
by studying analytic properties of some zeta functions of Igusa
type. Simple number theoretic proofs for some known properties are
also given.
\end{abstract}

\section{Introduction}

In a recent paper Liskovets \cite{Lis2010} investigated arithmetical
properties of the function
\begin{equation} \label{def_func_E}
E(m_1,\ldots,m_r):= \frac1{M}\sum_{k=1}^M
c_{m_1}(k)\cdots c_{m_r}(k),
\end{equation}
where $m_1,\ldots,m_r,M \in \N:=\{1,2,\ldots\}$ ($r\in \N$),
$m:=\lcm[m_1,\ldots,m_r]$, $m\mid M$ and $c_n(k)$ is the Ramanujan
sum defined as the sum of $k$-th powers of the primitive $n$-th
roots of unity ($k,n\in \N$), i.e.,
\begin{equation} \label{Ramanujan_sum}
c_n(k):= \sum_{\substack{1\le j \le n \\ \gcd(j,n)=1}} \exp(2\pi i
jk/n).
\end{equation}

The function $E$ given by \eqref{def_func_E} has been introduced by
Mednykh and Nedela \cite{MedNed2006} in order to handle certain
problems of enumerative combinatorics. The function
\eqref{def_func_E} has also arithmetical and topological
applications and it is called in \cite{Lis2010} as the
``orbicyclic'' arithmetic function.

For example, $E(m_1,\ldots,m_r)$ is the number of solutions
$(x_1,\ldots,x_r)\in \Z_M^r$ of the congruence
\begin{equation*} \label{lin_eq}
x_1+\ldots +x_r \equiv 0 \text{ (mod $M$)}
\end{equation*}
satisfying $\gcd(x_1,M)=M/m_1,\ldots,\gcd(x_r,M)=M/m_r$, see
\cite[Lemma 4.1]{MedNed2006}. It follows from this interpretation
that all the values $E(m_1,\ldots,m_r)$ are nonnegative integers.

Note that in case of one, respectively two variables,
\begin{equation} \label{one_variable_E}
E(m_1):= \frac1{M}\sum_{k=1}^M c_{m_1}(k)=
\begin{cases} 1, \  & m_1=1, \\  0, & \text{otherwise}, \end{cases}
\end{equation}
\begin{equation} \label{two_variable_E}
E(m_1,m_2):= \frac1{M}\sum_{k=1}^M c_{m_1}(k)c_{m_2}(k)=
\begin{cases} \phi(m), \  & m_1=m_2=m, \\  0, & \text{otherwise}, \end{cases}
\end{equation}
where $\phi$ is Euler's function. Formulae \eqref{one_variable_E}
and \eqref{two_variable_E} are well known properties of the
Ramanujan sums, \eqref{two_variable_E} being their orthogonality
property leading to the Ramanujan expansions of arithmetical
functions, see for example \cite{SchSpi1994}.

Another function, similar to $E$, is
\begin{equation} \label{A_multi}
A(m_1,\ldots,m_r):= \frac1{m} \sum_{k=1}^m \gcd(k,m_1)\cdots
\gcd(k,m_r),
\end{equation}
where $m_1,\ldots, m_r\in \N$ ($r\in \N$) and
$m:=\lcm[m_1,\ldots,m_r]$, as above.

The function \eqref{A_multi} was mentioned by Liskovets
\cite[section 4]{Lis2010} and it was considered by Deitmar, Koyama
and Kurokawa \cite{DeiKoyKur2008} in case $m_j\mid m_{j+1}$ ($1\le
j\le r-1$) by studying analytic properties of some zeta functions of
Igusa type. The explicit formula for the values $A(m_1,\ldots,m_r)$
derived in \cite[Section 3]{DeiKoyKur2008} was reproved by Minami
\cite{Min2009} for the general case $m_1,\ldots, m_r\in \N$, using
arguments of elementary probability theory. We remark that the
corresponding formulae of both papers \cite{DeiKoyKur2008,Min2009}
contain misprints.

For $r=1$ \eqref{A_multi} reduces to the function
\begin{equation} \label{A_arith_mean}
A(m):= \frac1{m} \sum_{k=1}^m \gcd(k,m)=\sum_{d\mid m}
\frac{\phi(d)}{d}
\end{equation}
giving the arithmetic mean of $\gcd(1,m), \ldots, \gcd(m,m)$. For
arithmetic and analytic properties of \eqref{A_arith_mean} and for a
survey of other gcd-sum-type functions of one variable see T\'oth
\cite{Tot2010}.

In the present paper we deduce new properties of the functions $E$
and $A$ and use them to give simple number theoretic proofs for some
of their known properties. We derive convolution-type identities for
$E$ and $A$ (Propositions \ref{prop_convo_repr_E} and
\ref{prop_convo_repr_A}) to show that they are multiplicative as
functions of several variables. We give other identities for $E$ and
$A$ (Propositions \ref{prop_sum_d_E} and \ref{prop_sum_d_A}) to
obtain explicit formulae for their values. We consider common
generalizations of these functions and point out that the function
$E$ and its connection to the number of solutions of certain linear
congruences occur in the literature in a slightly different form.

As an application of the identity of Proposition
\ref{prop_convo_repr_E} we give a simple direct proof of the
orthogonality property \eqref{two_variable_E} of the Ramanujan sums
(Application \ref{appl_ortho_Ramanujan}).

\section{Preliminaries} \label{section_prelim}

We present in this section some basic notions and properties needed
throughout the paper. For the prime power factorization of an
integer $n\in \N$ we will use the notation $n=\prod_p p^{\nu_p(n)}$,
where the product is over the primes $p$ and all but a finite number
of the exponents $\nu_p(n)$ are zero.

We recall that an arithmetic function of $r$ variables is a function
$f:\N^r \to \C$, notation $f\in {\cal F}_r$.  If $f,g\in {\cal
F}_r$, then their convolution is defined as
\begin{equation} \label{convo_functions}
(f*g)(m_1,\ldots,m_r)= \sum_{d_1\mid m_1, \ldots, d_r\mid m_r}
f(d_1,\ldots,d_r) g(m_1/d_1, \ldots, m_r/d_r).
\end{equation}

The set ${\cal F}_r$ forms a unital ring with ordinary addition and
convolution \eqref{convo_functions}, the unity being the function
$\varepsilon^{(r)}$ given by
\begin{equation} \label{function_unity}
\varepsilon^{(r)}(m_1,\ldots,m_r)= \begin{cases} 1, \  & m_1=\ldots
=m_r=1, \\  0, & \text{otherwise}. \end{cases}
\end{equation}

A function $f\in {\cal F}_r$ is invertible iff $f(1,\ldots,1)\ne 0$.
The inverse of the constant $1$ function is given by
$\mu^{(r)}(m_1,\ldots,m_r) = \mu(m_1) \ldots \mu(m_r)$, where $\mu$
is the M\"obius function.

A function $f\in {\cal F}_r$ is said to be multiplicative if it is
nonzero and
\begin{equation*} \label{def_mult}
f(m_1n_1,\ldots,m_rn_r)= f(m_1,\ldots,m_r) f(n_1,\ldots,n_r)
\end{equation*}
holds for any $m_1,\ldots,m_r,n_1,\ldots,n_r\in \N$ such that
$\gcd(m_1\cdots m_r,n_1\cdots n_r)=1$.

If $f$ is multiplicative, then it is determined by the values
$f(p^{a_1},\ldots,p^{a_r})$, where $p$ is a prime and
$a_1,\ldots,a_r\in \N_0:=\{0,1,2,\ldots\}$. More exactly,
$f(1,\ldots,1)=1$ and for any $m_1,\ldots,m_r\in \N$,
\begin{equation*}
f(m_1,\ldots,m_r)= \prod_p f(p^{\nu_p(m_1)}, \ldots,p^{\nu_p(m_r)}).
\end{equation*}

For example, the functions $(m_1,\ldots,m_r) \mapsto
\gcd(m_1,\ldots,m_r)$ and $(m_1,\ldots,m_r) \mapsto
\lcm[m_1,\ldots,m_r]$ are multiplicative. The function $\mu^{(r)}$
is also multiplicative.

The convolution \eqref{convo_functions} preserves the
multiplicativity of functions. Moreover, the multiplicative
functions form a subgroup of the group of invertible functions with
respect to convolution \eqref{convo_functions}.

These properties, which are well known in the one variable case,
follow easily from the definitions. For further properties of
arithmetic functions of several variables and of their ring we refer
to \cite{AlkZahZak2005}, \cite[Ch.\ VII]{Siv1989}.

In the one variable case $\1$, $\id_k$ ($k\in \N$), $\varepsilon$
and $\tau$ will denote the functions given by $\1(n)=1$,
$\id_k(n)=n^k$ ($n\in \N$), $\varepsilon(1)=1$, $\varepsilon(n)=0$
for $n>1$ and $\tau=\1*\1$ (divisor function), respectively.

The Ramanujan sums $c_n(k)$ can be represented as
\begin{equation} \label{Ramanujan_repr}
c_n(k)=\sum_{d \mid \gcd(k,n)} d\mu(n/d) \quad (k,n\in \N).
\end{equation}

Note that the function $n\mapsto c_n(k)$ is multiplicative for any
fixed $k\in \N$, and for any prime power $p^a$ ($a\in \N$),
\begin{equation} \label{c_n(k)}
c_{p^a}(k)=\begin{cases} p^a-p^{a-1}, \ \text{ if } \ p^a\mid k, \\
-p^{a-1}, \ \text{ if } \ p^{a-1}\mid k, p^a\nmid k, \\ 0, \ \text{
if } \ p^{a-1}\nmid k.
\end{cases}
\end{equation}

Also, $c_n(k)$ is multiplicative as a function of two variables,
i.e., considered as the function $c:\N^2\to \Z$, $c(k,n)=c_n(k)$.
The inequality $|c_n(k)|\le \gcd(k,n)$ holds for any $k,n\in \N$.

These and other general accounts of Ramanujan sums can be found in
the books by Apostol \cite{Apo1976}, McCarthy \cite{McC1986},
Schwarz and Spilker \cite{SchSpi1994}, Sivaramakrishnan
\cite{Siv1989}.

\section{The function $E$}

The following results were proved by Liskovets \cite{Lis2010}.

\begin{prop} {\rm (\cite[Lemmas 2, 5, Prop.\ 4]{Lis2010})} \label{prop_Liskovets}

(i) The function $E$ is multiplicative (as a function of several
variables).

(ii) Let $p^{a_1},\ldots, p^{a_r}$ be any powers of a prime $p$
($a_1,\ldots,a_r\in \N$). Assume that $a:=a_1=a_2=\ldots
=a_s>a_{s+1}\ge a_{s+2}\ge \ldots \ge a_r\ge 1$ ($r\ge s\ge 1$).
Then
\begin{equation}
E(p^{a_1},\ldots,p^{a_r})= p^{v} (p-1)^{r-s+1} h_{s}(p),
\end{equation}
where the integer $v$ is defined by $v=\sum_{j=1}^r a_j-r-a+1$ and
\begin{equation}
h_s(x)= \frac{(x-1)^{s-1}+(-1)^s}{x}
\end{equation}
is a polynomial of degree $s-2$ (for $s>1$).
\end{prop}

Note that Liskovets used the term semi-multiplicative function, but
this is reserved for another concept, see for example \cite[Ch.\ XI]{Siv1989}.

\begin{cor} {\rm (\cite[Th.\ 8, Cor.\ 11]{Lis2010})} \label{cor_Liskovets}

(i) For any integers $m_1,\ldots,m_r\in \N$,
\begin{equation}
E(m_1,\ldots,m_r)=\prod_{p\mid m} p^{v(p)} (p-1)^{r(p)-s(p)+1}
h_{s(p)}(p),
\end{equation}
where $v(p)$ and $s(p)$, depending now on $p$ are the integers $v$
and $s$, respectively, defined in Proposition \ref{prop_Liskovets}.

(ii) For $m_1=\ldots =m_r=m$,
\begin{equation} \label{formula_E_equal_variables}
f_r(m):= E(m,\ldots,m)= m^{r-1} \prod_{p\mid m} \frac{(p-1)
h_{r}(p)}{p^{r-1}}.
\end{equation}
\end{cor}

We first give the following simple convolution representation for
the function $E$.

\begin{prop} \label{prop_convo_repr_E} For any $m_1,\ldots,m_r\in \N$,
\begin{equation} \label{convo_repr_E}
E(m_1,\ldots,m_r)= \sum_{d_1\mid m_1, \ldots, d_r\mid m_r}
\frac{d_1\cdots d_r}{\lcm[d_1,\ldots, d_r]} \mu(m_1/d_1)\cdots
\mu(m_r/d_r).
\end{equation}
\end{prop}

\begin{proof} Using formula \eqref{Ramanujan_repr},
\begin{equation*}
E(m_1,\ldots,m_r)= \frac1{M} \sum_{k=1}^M \sum_{d_1 \mid
\gcd(k,m_1)} d_1\mu(m_1/d_1) \cdots \sum_{d_r \mid \gcd(k,m_r)}
d_r\mu(m_r/d_r)
\end{equation*}
\begin{equation*}
= \frac1{M} \sum_{d_1\mid m_1, \ldots, d_r\mid m_r} d_1\mu(m_1/d_1)
\cdots d_r\mu(m_r/d_r) \sum_{1\le k\le M, d_1\mid k,\ldots d_r\mid
k} 1,
\end{equation*}
where the inner sum is $\ds \sum_{1\le k\le M,
\lcm[d_1,\ldots,d_r]\mid k} 1= M/\lcm[d_1,\ldots, d_r]$, ending the
proof.
\end{proof}

By M\"obius inversion we obtain from \eqref{convo_repr_E},

\begin{cor} For any $m_1,\ldots,m_r\in \N$,
\begin{equation} \label{sum_E}
\sum_{d_1\mid m_1, \ldots, d_r\mid m_r} E(d_1,\ldots,d_r)=
\frac{m_1\cdots m_r}{\lcm[m_1,\ldots, m_r]}.
\end{equation}
\end{cor}

Formula \eqref{convo_repr_E} has also a number of other
applications:

\begin{appl} {\rm Formula \eqref{convo_repr_E} shows that $E$ is integral valued and that it
does not depend on $M$, so in \eqref{def_func_E} one can take $M=m$,
the lcm of $m_1,\ldots,m_r$, remarked also by Liskovets
\cite[Section 1]{Lis2010}.}
\end{appl}

\begin{appl} {\rm If $m_1,\ldots,m_r$ are pairwise relatively prime, then
$E(m_1,\ldots,m_r)= \varepsilon^{(r)}(m_1,\ldots,m_r)$, defined by
\eqref{function_unity}. Indeed, in this case
$\lcm[d_1,\ldots,d_r]=d_1\cdots d_r$ for any $d_i\mid m_i$ ($1\le
i\le r$) and the claim follows from \eqref{convo_repr_E} using that
$\sum_{d\mid n} \mu(d) =\varepsilon(n)$.}
\end{appl}

\begin{appl} {\rm Also, \eqref{convo_repr_E} furnishes a simple direct proof of the
multiplicativity of $E$. Note that Liskovets \cite{Lis2010} used
other arguments to show the multiplicativity. Observe that,
according to \eqref{convo_repr_E}, $E$ is the convolution of the
functions $F$ and $\mu^{(r)}$, where $F$ is given by
\begin{equation*} \label{F}
F(m_1,\ldots,m_r) = \frac{m_1\cdots m_r}{\lcm[m_1,\ldots,m_r]}.
\end{equation*}

Since $F$ and $\mu^{(r)}$ are multiplicative, $E$ is multiplicative
too.}
\end{appl}

\begin{appl} \label{appl_ortho_Ramanujan} {\rm Formula \eqref{convo_repr_E}
leads to a simple direct proof of the
orthogonality property \eqref{two_variable_E}. Using the Gauss
formula $\sum_{d\mid n} \phi(d)=n$,
\begin{equation*}
E(m_1,m_2):= \frac1{M}\sum_{k=1}^M c_{m_1}(k)c_{m_2}(k)=
\sum_{d_1\mid m_1, d_2\mid m_2} \frac{d_1d_2 \mu(m_1/d_1)
\mu(m_2/d_2)}{\lcm[d_1,d_2]}
\end{equation*}
\begin{equation*}
= \sum_{d_1\mid m_1, d_2\mid m_2} \mu(m_1/d_1) \mu(m_2/d_2)
\gcd(d_1,d_2)= \sum_{d_1\mid m_1, d_2\mid m_2} \mu(m_1/d_1)
\mu(m_2/d_2) \sum_{\delta \mid \gcd(d_1,d_2)} \phi(\delta)
\end{equation*}
\begin{equation*}
= \sum_{\delta ak=m_1, \delta b\ell= m_2} \mu(k) \mu(\ell)
\phi(\delta) = \sum_{\delta u=m_1, \delta v=m_2} \phi(\delta)
\sum_{ak=u} \mu(k) \sum_{b\ell=v} \mu(\ell),
\end{equation*}
where one of the inner sums are zero, unless $u=v=1$ and obtain that
$E(m_1,m_2)= \phi(m)$ for $m_1=m_2=m$ and $E(m_1,m_2)=0$ otherwise.}
\end{appl}

Now we derive another identity for $E$ which furnishes an
alternative proof of formula (ii) in Proposition
\ref{prop_Liskovets}.

\begin{prop} \label{prop_sum_d_E} For any $m_1,\ldots,m_r\in \N$,
\begin{equation} \label{another_repr_E}
E(m_1,\ldots,m_r)= \frac1{m}\sum_{d\mid m} c_{m_1}(d)\cdots
c_{m_r}(d)\phi(m/d).
\end{equation}
\end{prop}

\begin{proof} It follows from \eqref{Ramanujan_repr} that $c_n(k)=c_n(\gcd(k,n))$
($k,n\in \N$). Observe that for any $i\in \{1,\ldots,r\}$, $
\gcd(\gcd(k,m),m_i)=\gcd(k,\gcd(m,m_i)) =\gcd(k,m_i)$, since
$m_i\mid m$, hence $c_{m_i}(k)=c_{m_i}(\gcd(k,m_i))=
c_{m_i}(\gcd(\gcd(k,m),m_i))=c_{m_i}(\gcd(k,m))$. We obtain
\begin{equation*} \label{gcd_repr_func_E}
E(m_1,\ldots,m_r)= \frac1{m}\sum_{k=1}^m c_{m_1}(\gcd(k,m))\cdots
c_{m_r}(\gcd(k,m)),
\end{equation*}
and by grouping the terms according to the values $\gcd(k,m)=d$,
where $d\mid m$, $k=dj$, $1\le j\le m/d$, $\gcd(j,m/d)=1$, we obtain
\eqref{another_repr_E}.
\end{proof}

\begin{appl} {\rm By \eqref{another_repr_E}, with the notation of Proposition
\ref{prop_Liskovets},
\begin{equation} \label{c_sum}
E(p^{a_1},\ldots,p^{a_r})= \frac1{p^a}\sum_{d\mid p^a}
c_{p^{a_1}}(d)\cdots c_{p^{a_r}}(d) \phi(p^a/d).
\end{equation}

Using \eqref{c_n(k)} we see that only two terms of \eqref{c_sum} are
nonzero, namely those for $d=p^a$ and $d=p^{a-1}$. Hence,
\begin{equation*}
E(p^{a_1},\ldots,p^{a_r})= \frac1{p^a} \left(c_{p^{a_1}}(p^a)\cdots
c_{p^{a_r}}(p^a) \phi(1)+ c_{p^{a_1}}(p^{a-1})\cdots
c_{p^{a_r}}(p^{a-1}) \phi(p)\right)
\end{equation*}
\begin{equation*}
= \frac1{p^a} \left( (p-1)p^{a_1-1} \cdots (p-1)p^{a_r-1} +
(-p^{a-1})^s (p-1)p^{a_{s+1}-1} \cdots (p-1)p^{a_r-1} (p-1)\right),
\end{equation*}
and a short computation gives formula (ii) in Proposition
\ref{prop_Liskovets}.}
\end{appl}

Consider the function $f_r(m)$ defined in Corollary
\ref{cor_Liskovets} (case $m_1=\ldots =m_r=m$). Here
$f_1=\varepsilon$, $f_2=\phi$.

\begin{prop} \label{average_order}
Let $r\ge 3$. The average order of the function $f_r(m)$ is
$C_rm^{r-1}$, where
\begin{equation}
C_r:=  \prod_p \left(1+ \frac{(p-1) h_{r}(p)-p^{r-1}}{p^r}\right).
\end{equation}

More exactly, for any $0< \varepsilon < 1$,
\begin{equation} \label{asymp_E}
\sum_{m\le x} f_r(m)= \frac{C_r}{r}x^{r} +{\cal
O}(x^{r-1+\varepsilon}).
\end{equation}
\end{prop}

\begin{proof} The function $f_r$ is multiplicative and by
\eqref{formula_E_equal_variables},
\begin{equation*}
\sum_{m=1}^{\infty} \frac{f_r(m)}{m^s} = \zeta(s-r+1) \prod_p \left(
1+\frac{(p-1)h_r(p)-p^{r-1}}{p^s}\right)
\end{equation*}
for $s\in \C$, $\RE s>r$, where the infinite product is absolutely
convergent for $\RE s>r-1$. Hence $f_r=g_r* \id_{r-1}$ in terms of
the Dirichlet convolution, where $g_r$ is multiplicative and for any
prime power $p^a$ ($a\in \N$),
\begin{equation*}
g_r(p^a)= \begin{cases} (p-1)h_r(p)-p^{r-1}, & a=1, \\
0, & a\ge 2.
\end{cases}
\end{equation*}

We obtain
\begin{equation*}
\sum_{m\le x} f_r(m) = \sum_{d\le x} g_r(d) \sum_{e\le x/d} e^{r-1}
= \frac{x^r}{r} \sum_{d\le x} \frac{g_r(d)}{d^r} +{\cal O}\left(
x^{r-1} \sum_{d\le x} \frac{|g_r(d)|}{d^{r-1}} \right)
\end{equation*}
and \eqref{asymp_E} follows by usual estimates.
\end{proof}

\section{The function $A$}

Consider now the function $A$ given by \eqref{A_multi}.

The next formulae are similar to \eqref{convo_repr_E} and
\eqref{another_repr_E}. The following one was already given in
\cite[Section 3]{Tot2010}.

\begin{prop} \label{prop_convo_repr_A} For any $m_1,\ldots,m_r\in \N$,
\begin{equation} \label{convo_repr_A}
A(m_1,\ldots,m_r)= \sum_{d_1\mid m_1, \ldots, d_r\mid m_r}
\frac{\phi(d_1)\cdots \phi(d_r)}{\lcm[d_1,\ldots, d_r]},
\end{equation}
\end{prop}

\begin{proof} Similar to the proof of Proposition \ref{prop_convo_repr_E},
this is obtained by inserting $\gcd(k,m_i)=\sum_{d_i\mid
\gcd(k,m_i)} \phi(d_i)$ ($1\le i\le r$).
\end{proof}

\begin{cor} The function $A$ is multiplicative (as a function of
several variables).
\end{cor}

\begin{proof} According to
\eqref{convo_repr_A}, $A$ is the convolution of the functions $G$
and the constant $1$ function, where $G$ is given by
\begin{equation*} \label{G}
G(m_1,\ldots,m_r) = \frac{\phi(m_1)\cdots
\phi(m_r)}{\lcm[m_1,\ldots,m_r]},
\end{equation*}
both being multiplicative. Therefore $A$ is also multiplicative.
\end{proof}

\begin{prop} \label{prop_sum_d_A} For any $m_1,\ldots,m_r\in \N$,
\begin{equation} \label{sum_d_A}
A(m_1,\ldots,m_r)= \frac1{m} \sum_{d\mid m} \gcd(d,m_1) \cdots
\gcd(d,m_r) \phi(m/d).
\end{equation}
\end{prop}

\begin{proof} Similar to the proof of Proposition
\ref{prop_sum_d_E}. Using that $\gcd(k,m_i)= \gcd(\gcd(k,m),m_i)$
($1\le i\le r$) we have
\begin{equation*}
A(m_1,\ldots,m_r)= \frac1{m}\sum_{k=1}^m \gcd(\gcd(k,m),m_1)\cdots
\gcd(\gcd(k,m),m_r),
\end{equation*}
and by grouping the terms according to the values $\gcd(k,m)=d$ we
obtain the formula.
\end{proof}

\begin{cor} Let $p$ be a prime and let $a_1,\ldots,a_r\in \N$. Assume that $a_0:=0< a_1\le
a_2\le \ldots \le a_r$. Then
\begin{equation}
A(p^{a_1},\ldots,p^{a_r}) = p^{a_0+a_1+\ldots +a_{r-1}} +
\left(1-\frac1{p}\right) \sum_{\ell=1}^r p^{a_0+a_1+\ldots
+a_{\ell-1}}\sum_{j=a_{\ell-1}}^{a_{\ell}-1} p^{(r-\ell)j}.
\end{equation}
\end{cor}

\begin{proof} Let $a:=a_r$. Then $\lcm[p^{a_1},\ldots,p^{a_r}]=p^a$. From \eqref{sum_d_A} we
obtain
\begin{equation*}
A(p^{a_1},\ldots,p^{a_r}) =\frac1{p^a} \sum_{j=0}^a
\gcd(p^j,p^{a_1}) \cdots \gcd(p^j,p^{a_r}) \phi(p^{a-j})
\end{equation*}
\begin{equation*}
= p^{a_0+a_1+\ldots +a_{r-1}} + \left(1-\frac1{p}\right)
\sum_{j=0}^{a-1} p^{\min(j,a_1)+\ldots +\min(j,a_r)-j},
\end{equation*}
where the last sum is
\begin{equation*}
\sum_{j=0}^{a_1-1} p^{rj-j}+ \sum_{j=a_1}^{a_2-1} p^{a_1+(r-1)j-j}+
\ldots  +\sum_{j=a_{r-1}}^{a_r-1} p^{a_1+a_2+\ldots +a_{r-1}+ 1j-j}
\end{equation*}
\begin{equation*}
=\sum_{\ell=1}^r \sum_{j=a_{\ell-1}}^{a_{\ell}-1} p^{a_0+a_1+\ldots
+a_{\ell-1}+ (r-\ell)j},
\end{equation*}
finishing the proof.
\end{proof}

\begin{appl} {\rm From \eqref{convo_repr_A} we have for any $m_1,\ldots, m_r\in \N$,
\begin{equation}
A(m_1,\ldots, m_r)\ge \sum_{d_1\mid m_1, \ldots, d_r\mid m_r}
\frac{\phi(d_1)\cdots \phi(d_r)}{d_1\cdots d_r} = \sum_{d_1\mid m_1}
\frac{\phi(d_1)}{d_1} \cdots \sum_{d_r\mid m_r}
\frac{\phi(d_r)}{d_r} = A(m_1)\cdots A(m_r),
\end{equation}
cf. \eqref{A_arith_mean}, with equality if $m_1,\ldots, m_r$ are
pairwise relatively prime.}
\end{appl}

Note that if $m_1=\ldots =m_r=m$, then from its definition,
\begin{equation}
A_r(m):= A(m,\ldots,m) =\frac1{m} \sum_{k=1}^m (\gcd(k,m))^r=
\frac1{m} \sum_{d\mid m} d^r\phi(m/d),
\end{equation}
which is a multiplicative function. An asymptotic formula for
$\sum_{m\le x} A_r(m)$ was given by Alladi \cite{All1975}. See also
\cite[Section 2]{Tot2010}.

A simple inequality for the functions $E$ and $A$ is given by

\begin{prop} For any $m_1,\ldots,m_r\in \N$,
\begin{equation}
E(m_1,\ldots,m_r)\le A(m_1,\ldots,m_r).
\end{equation}
\end{prop}

\begin{proof} Using the inequality $|c_n(k)|\le \gcd(k,n)$, mentioned in the
Introduction, we obtain
\begin{equation*}
E(m_1,\ldots,m_r)\le \frac1{m} \sum_{k=1}^m |c_{m_1}(k)|\cdots
|c_{m_r}(k)| \le \frac1{m} \sum_{k=1}^n \gcd(k,m_1) \cdots
\gcd(k,m_r)= A(m_1,\ldots,m_r).
\end{equation*}
\end{proof}

\section{Generalizations}

Let $f\in {\cal F}_2$ be a function of two variables and consider
the function
\begin{equation}
F_f(m_1,\ldots,m_r):= \frac1{m} \sum_{k=1}^m f(k,m_1)\cdots
f(k,m_r).
\end{equation}

\begin{prop} {\rm (\cite[Lemma 5]{Lis2010})} If $n\mapsto f(k,n)$ is
multiplicative for any $k\in \N$ and $k\mapsto f(k,n)$ is periodic
(mod $n$) for any $n\in \N$, then $F_f$ is multiplicative.
\end{prop}

Now suppose that $f$ has the following representation:
\begin{equation} \label{gen_Apostol_type}
f(k,n) = \sum_{d\mid \gcd(k,n)} g(d)h(n/d) \quad (k,n\in \N)
\end{equation}
where $g,h \in {\cal F}_1$ are arbitrary functions. Functions $f$
defined in this way, as generalizations of the Ramanujan sums, were
investigated in \cite{AndApo1953, Apo1972}. See also Apostol
\cite[Section 8.3]{Apo1976}.

\begin{prop} Assume that $f$ is given by \eqref{gen_Apostol_type}.
Then

i) $F_f$ has the representations
\begin{equation} \label{convo_repr_gen}
F_f(m_1,\ldots,m_r)= \sum_{d_1\mid m_1, \ldots, d_r\mid m_r}
\frac{g(d_1)\cdots g(d_r)}{\lcm[d_1,\ldots, d_r]} h(m_1/d_1)\cdots
h(m_r/d_r),
\end{equation}
\begin{equation} \label{divisor_repr_gen}
F_f(m_1,\ldots,m_r)= \frac1{m} \sum_{d\mid m} f(\gcd(d,m_1)) \cdots
f(\gcd(d,m_r)) \phi(m/d).
\end{equation}

ii) If $g$ and $h$ are multiplicative functions, then $F_f$ is
multiplicative (as a function of several variables).
\end{prop}

\begin{proof} i) Follows by the same arguments as in the proofs of
Propositions \ref{prop_convo_repr_E}, \ref{prop_sum_d_E},
\ref{prop_convo_repr_A} and \ref{prop_sum_d_A}.

ii) Direct consequence of \eqref{convo_repr_gen}.
\end{proof}

The functions $E$ and $A$ are recovered choosing $g=\id$, $h=\mu$
and $g=\phi$, $h=\1$, respectively, where note that
$\gcd(k,n)=\sum_{d\mid \gcd(k,n)} \phi(d)$.

Now let $f(k,n)=\overline{f}(\gcd(k,n))$ ($k,n\in \N$), where
$\overline{f}\in {\cal F}_1$ is an arbitrary function. Then $f$ is
of type \eqref{gen_Apostol_type} with $g=\overline{f}*\mu$, $h=\1$,
since $\overline{f}(\gcd(k,n))=\sum_{d\mid \gcd(k,n)}
(\overline{f}*\mu)(d)$.

Special choices of $\overline{f}$ can be considered. For
$\overline{f}=\id$ we reobtain the function $A$. As another example,
let $\overline{f}=\tau$, with $g=h=\1$. Then we obtain

\begin{cor} The function $F_{\tau}\in {\cal F}_r$ is multiplicative and
\begin{equation} \label{convo_repr_gen_tau}
F_{\tau}(m_1,\ldots,m_r)= \sum_{d_1\mid m_1, \ldots, d_r\mid m_r}
\frac1{\lcm[d_1,\ldots, d_r]},
\end{equation}
\begin{equation} \label{divisor_repr_gen_tau}
F_f(m_1,\ldots,m_r)= \frac1{m} \sum_{d\mid m} \tau(\gcd(d,m_1))
\cdots \tau(\gcd(d,m_r)) \phi(m/d).
\end{equation}
\end{cor}

Further common generalizations of the functions $E$ and $A$ can be
given using $r$-even functions. See \cite{McC1986,SchSpi1994,TotHau}
for their definitions and properties.

\section{Linear congruences with constraints}

A direct generalization of the interpretation of $E$ given in the
Introduction is the following. Let $M\in \N$ and let ${\cal D}_k(M)$
($1\le k\le r)$ be arbitrary nonempty subsets of the set of
(positive) divisors of $M$. For an integer $n$ let $N_n(M,{\cal
D}_1,\ldots,{\cal D}_r)$ stand for the number of solutions
$(x_1,\ldots,x_r) \in \Z^r_M$ of the congruence
\begin{equation} \label{gen_lin_eq}
x_1+\ldots +x_r \equiv n \text{ (mod $M$)}
\end{equation}
satisfying $\gcd(x_1,M)\in {\cal D}_1,\ldots,\gcd(x_r,M)\in {\cal
D}_r$.

If ${\cal D}_k=\{M/m_k\}$ ($1\le k\le r)$ and $n=0$, then we
reobtain the function $E$.

Special cases of the function $N_n(M,{\cal D}_1,\ldots,{\cal D}_r)$
were investigated earlier by several authors.

The case ${\cal D}_k=\{1\}$, i.e., $\gcd(x_k,M)=1$ ($1\le k\le r)$
was considered for the first time by Rademacher \cite{Rad1925} in
1925 and Brauer \cite{Bra1926} in 1926. It was recovered by Nicol
and Vandiver \cite{NicVan1954} in 1954, Cohen \cite{Coh1955} in
1955, Rearick \cite{Rea1963} in 1963, and others. The case $r=2$ was
treated by Alder \cite{Ald1958} in 1958.

The general case of arbitrary subsets ${\cal D}_k$ was investigated,
among others, by McCarthy \cite {McC1975} in 1975 and by Spilker
\cite{Spi1996} in 1996. One has
\begin{equation} \label{formula_gen_lin_eq}
N_n(M,{\cal D}_1,\ldots,{\cal D}_r)=\frac1{M}\sum_{d\mid M}
c_{M/d}(n) \prod_{i=1}^r \sum_{e\in {\cal D}_i(M)} c_{M/e}(d).
\end{equation}

For $M=m$, ${\cal D}_k=\{m/m_k\}$ ($1\le k\le r)$ and $n=0$
\eqref{formula_gen_lin_eq} reduces to our Proposition
\ref{prop_sum_d_E}.

The proof of formula \eqref{formula_gen_lin_eq} given in
\cite[Section 4]{Spi1996}, see also \cite[Ch.\ 3]{McC1986}, uses
properties of $r$-even functions, Cauchy products and
Ramanujan--Fourier expansions of functions.

See Chapter 3 of the book of McCarthy \cite{McC1986} for a survey of
this topic.

It is well known that the number of solutions of polynomial
congruences can be expressed using exponential sums, see for ex.
\cite[Th.\ 1.31]{Nar1983}. Although the obtained expression can be
easily transformed by means of Ramanujan sums in case of linear
congruences with side conditions, this is not used in the literature
cited in this Section.

In what follows we give a simple direct proof of
\eqref{formula_gen_lin_eq} in the case ${\cal D}_k(M)=\{d_k\}$
($1\le k\le r$), applying the method of above.

\begin{prop} \label{prop_general_cong} Let $M\in \N$, $n\in \Z$ and let $d_1,\ldots,d_r\mid M$. Then
\begin{equation}
N_n(M,\{d_1\},\ldots,\{d_r\})= \frac1{M} \sum_{k=1}^M c_{M/d_1}(k)
\cdots c_{M/d_r}(k) \exp(-2\pi ikn/M)
\end{equation}
\begin{equation}
= \frac1{M} \sum_{\delta \mid M} c_{M/d_1}(\delta)\cdots
c_{M/d_r}(\delta) c_{M/\delta}(n).
\end{equation}
\end{prop}

\begin{proof} Only the simple fact
\begin{equation} \label{exp_sum}
\sum_{k=1}^M \exp(2\pi ikn/M)=\begin{cases} M, \  & M\mid n, \\  0,
& M\nmid n. \end{cases}
\end{equation}
and the definition of Ramanujan sums are required. By the definition
of $N_n(M,\{d_1\},\ldots,\{d_r\})$,
\begin{equation*}
N:= N_n(M,\{d_1\},\ldots,\{d_r\}) = \frac1{M} \sum_{\substack{1\le x_1\le M\\
\gcd(x_1,M)=d_1}} \cdots \sum_{\substack{1\le x_r\le M\\
\gcd(x_1,M)=d_r}} \sum_{k=1}^n  \exp(2\pi ik(x_1+\ldots +x_r-n)/M)
\end{equation*}
\begin{equation*}
= \frac1{M} \sum_{k=1}^M  \exp(-2\pi ikn/M) \sum_{\substack{1\le x_1\le M\\
\gcd(x_1,M)=d_1}} \exp(2\pi ik x_1/M) \cdots \sum_{\substack{1\le x_r\le M\\
\gcd(x_r,M)=d_r}} \exp(2\pi ik x_r/M),
\end{equation*}
and denoting $x_k=d_ky_k$, $\gcd(y_k,M/d_k)=1$ ($1\le k\le M$),
\begin{equation*}
N = \frac1{M} \sum_{k=1}^M \exp(-2\pi ikn/M) \sum_{\substack{1\le y_1\le M/d_1\\
\gcd(y_1,M/d_1)=1}} \exp(2\pi ik y_1/(M/d_1)) \cdots \sum_{\substack{1\le y_r\le M/d_r\\
\gcd(y_r,M/d_r)=1}} \exp(2\pi ik y_r/(M/d_r))
\end{equation*}
\begin{equation*}
= \frac1{M} \sum_{k=1}^M \exp(-2\pi ikn/M) c_{M/d_1}(k)\cdots
c_{M/d_r}(k)
\end{equation*}

To obtain the second formula use that $c_{M/d_i} (k)=
c_{M/d_i}(\gcd(k,M))$ and group the terms according to the values of
$\gcd(k,M)=\delta$, cf. proof of Proposition \ref{prop_sum_d_E}.
\end{proof}

For $n=0$ and $d_k=M/m_k$ ($1\le k\le r$) we reobtain the
interpretation of the function $E$ and formula
\eqref{another_repr_E}. We remark that the proof of Proposition
\ref{prop_general_cong} is a slight simplification of the proof of
\cite[Lemma 4.1]{MedNed2006}.


\begin{thebibliography}{99}

\bibitem{Ald1958} {H.~L.~Alder}, A generalization of the Euler $\phi$--function, {\it Amer.
Math. Monthly} {\bf 65} (1958) 690--692.

\bibitem{AlkZahZak2005} E.~Alkan, A.~Zaharescu and M.~Zaki, Arithmetical
functions in several variables, {\it Int. J. Number Theory} {\bf 1}
(2005), 383--399.

\bibitem{All1975} K.~Alladi, On generalized Euler functions
and related totients, in {\it New Concepts in Arithmetic Functions},
Matscience Report 83, Madras, 1975.

\bibitem{AndApo1953} D.~R.~Anderson, T.~M.~Apostol, The evaluation of Ramanujan's
sums and generalizations, {\it Duke Math. J.}, {\bf 20} (1953),
211--216.

\bibitem{Apo1972} T.~M.~Apostol, Arithmetical properties of generalized Ramanujan sums,
{\it Pacific J. Math.} {\bf 41} (1972), 281--293.

\bibitem{Apo1976} T.~M.~Apostol, {\it Introduction to Analytic Number Theory},
Sprin\-ger, 1976.

\bibitem{Bra1926} A.~Brauer, L\"osung der Aufgabe 30, Jber. Deutsch.
Math.--Verein {\bf 35} (1926), 92--94.

\bibitem{Coh1955} E.~Cohen, A class of arithmetical functions, {\it Proc. Nat. Acad.
Sci. U.S.A.} {\bf 41} (1955), 939--944.

\bibitem{DeiKoyKur2008} A.~Deitmar, S.~Koyama and N.~Kurokawa, Absolute
zeta functions, {\it Proc. Japan Acad. Ser. A Math. Sci.} {\bf 84}
(2008), 138--142.

\bibitem{Lis2010} V.~A.~Liskovets, A multivariate arithmetic function
of combinatorial and topological significance, {\it Integers} {\bf
10} (2010), 155--177.

\bibitem{McC1975} P.~J.~McCarthy, The number of restricted solutions of some systems of
linear congruences, {\it Rend. Sem. Mat. Univ. Padova} {\bf 54}
(1975), 59--68.

\bibitem{McC1986} P.~J.~McCarthy, {\it Introduction to Arithmetical
Functions}, Uni\-ver\-si\-text, Sprin\-ger, 1986.

\bibitem{MedNed2006} A.~D.~Mednykh and R.~Nedela, Enumeration of unrooted maps with
given genus, {\it J. Combin. Theory}, Ser. B {\bf 96} (2006),
706--729.

\bibitem{Min2009} N.~Minami, On the random variable $\N \ni l \mapsto \gcd(l,n_1) \gcd(l,
n_2)\cdots \gcd(l,n_k)\in \N$, Preprint, 2009,
\url{http://arxiv.org/abs/0907.0918v2}

\bibitem{Nar1983} W. Narkiewicz, {\it Number Theory}, World Scientific, Singapore, 1983.

\bibitem{NicVan1954} C.~A.~Nicol, H.~S.~Vandiver, A von Sterneck arithmetical function
and restricted partitions with respect to a modulus, {\it Proc. Nat.
Acad. Sci. U.S.A.} {\bf 40} (1954), 825--835.

\bibitem{Rad1925} H.~Rademacher, Aufgabe 30, Jber. Deutsch.
Math.--Verein {\bf 34} (1925), 158.

\bibitem{Rea1963} D.~Rearick, A linear congruence with side conditions, {\it Amer. Math. Monthly}
{\bf 70} (1963), 837--840.

\bibitem{SchSpi1994} W.~Schwarz and J.~Spilker, {\it Arithmetical Functions},
London Mathematical Society Lecture Note Series, 184, Cambridge
University Press, 1994.

\bibitem{Siv1989} R.~Sivaramakrishnan, {\it Classical Theory of Arithmetic Functions},
Monographs and Textbooks in Pure and Applied Mathematics, Vol.\ 126,
Marcel Dekker, 1989.

\bibitem{Spi1996} J.~Spilker, Eine einheitliche Methode zur Behandlung einer
linearen Kogruenz mit Nebenbedingungen, {\it Elem. Math.} {\bf 51}
(1996), 107--116.

\bibitem{Tot2010} L.~T\'oth, A survey of gcd--sum functions, {\it J. Integer Sequences}
{\bf 13} (2010), Article 10.8.1, 23 pp.

\bibitem{TotHau} L.~T\'oth and P.~Haukkanen, The discrete Fourier transform of $r$-even functions,
2010, submitted, \url{http://arxiv.org/abs/1009.5281v1}

\end{thebibliography}
\end{document}